\documentclass[11pt]{amsart}

\usepackage{amsmath}
\usepackage{amssymb}
\usepackage{latexsym}
\usepackage{amscd}
\usepackage{citesort}
\usepackage{mathrsfs}
\usepackage[all]{xy}

\newdimen\AAdi%
\newbox\AAbo%
\def\AAk#1#2{\s_etbox\AAbo=\hbox{#2}\AAdi=\wd\AAbo\kern#1\AAdi{}}%
\def\AAr#1#2#3{\s_etbox\AAbo=\hbox{#2}\AAdi=\ht\AAbo\raise#1\AAdi\hbox{#3}}%
\font\tenmsb=msbm10 at 12pt \font\sevenmsb=msbm7 at 8pt
\font\fivemsb=msbm5 at 6pt
\newfam\msbfam
\textfont\msbfam=\tenmsb \scriptfont\msbfam=\sevenmsb
\scriptscriptfont\msbfam=\fivemsb
\def\Bbb#1{{\tenmsb\fam\msbfam#1}}
\newtheorem{theorem}{Theorem}

\newtheorem{remark}[theorem]{Remark}

\newtheorem{lemma}[theorem]{Lemma}

\numberwithin{equation}{section} \numberwithin{theorem}{section}

\renewcommand{\topmargin}{0cm}
\renewcommand{\oddsidemargin}{5mm}
\renewcommand{\evensidemargin}{5mm}
\renewcommand{\textwidth}{150mm}
\renewcommand{\textheight}{230mm}

\def\R{\mathbb R}

\def\S{\mathbb S}
\def\n{\mathbf n}

\def\sr#1{\mathscr{#1}}
\def\ir#1{\mathbb R^{#1}}

\def\cc#1{\Bbb C^{#1}}
\def\f#1#2{\frac{#1}{#2}}

\def\grs#1#2{\bold G_{#1,#2}}

\def\dt#1{\frac {d\,#1}{d\,t}}

\def\a{\alpha}
\def\be{\beta}

\def\r{\Re_{I\!V}}

\def\p#1{\partial #1}

\def\de{\delta}
\def\De{\Delta}
\def\e{\eta}

\def\G{\Gamma}
\def\g{\gamma}
\def\k{\kappa}
\def\n{\nabla}
\def\la{\lambda}
\def\La{\Lambda}
\def\lan{\langle}
\def\ran{\rangle}

\def\Om{\Omega}
\def\th{\theta}

\def\r{\rho}

\begin{document}

\title
{The Rigidity Theorems for Lagrangian Self shrinkers}

\author{Qi Ding}\author{Y. L. Xin}
\address{Institute of Mathematics, Fudan University,
Shanghai 200433, China} \email{09110180013@fudan.edu.cn}
\email{ylxin@fudan.edu.cn}
\thanks{The research was partially supported by NSFC}
\begin{abstract}
By the integral method we prove that  any space-like entire graphic
self-shrinking solution to Lagrangian mean curvature flow in
$\R^{2n}_{n}$ with the indefinite metric $\sum_i dx_idy_i$ is flat.
This result improves the previous ones in \cite{HW} and \cite{CCY}
by removing the additional assumption in their results. In a similar
manner, we reprove its Euclidean counterpart which is established in
\cite{CCY}.
\end{abstract}

\maketitle

\section{Introduction}
Let $M$ be a submanifold in $\ir{m+n}.$ Mean curvature flow is a one-parameter family $X_t=X(\cdot,
t)$ of immersions $X_t:M\to \ir{m+n}$ with corresponding images
$M_t=X_t(M)$ such that
\begin{equation*}\left\{\begin{split}
\dt{}X(x, t)&=H(x, t),\qquad x\in M\\
X(x, 0)&=X(x)
\end{split}\right.
\end{equation*}
is satisfied, where $H(x, t)$ is the mean curvature vector of
$M_t$ at $X(x, t)$ in $\ir{m+n}.$

An important class of solutions to the above mean curvature flow
equations are self-similar shrinkers, whose profiles,
self-shrinkers, satisfy a system of quasi-linear elliptic PDE of the
second order
\begin{equation}\label{SS}
H = -\frac{X^N}{2},
\end{equation}
where $(\cdots)^N$ stands for the orthogonal projection into the normal bundle $NM$.

In the ambient pseudo-Euclidean space we can also study the mean
curvature flow (see \cite{E1} \cite{E2}  \cite{E3}  \cite{X1} and
\cite{H}, for example). And self-shrinking graphs with high
codimensions in pseudo-Euclidean space has been studied in
\cite{DW}. Let $\R^{2n}_n$ be Euclidean space with null coordinates
$(x, y) = (x_1, \cdots, x_n; \, y_1, \cdots, y_n)$, which means that
the indefinite metric is defined by $ds^2=\sum_i dx_i dy_i.$ If
$M=\{(x,Du(x))\big|\ x\in\R^n\}$ is a space-like submanifold in
$\R^{2n}_n$, then $u$ is convex (In this paper, we say that a smooth function $f$ is \emph{convex}, if $D^2f>0$, i.e., hessian of $f$ is
positive definite in $\ir{n}$). The underlying Euclidean space
$\ir{2n}=\cc{n}$ of $\ir{2n}_n$ has the usual complex structure. It
is easily seen that $M$ is a Lagrangian submanifold in $\ir{2n}$
(\cite{X},\  Lemma 5.2.11), as well as in $\ir{2n}_n$ . Moreover, if
$M$ is also a self-shrinker, namely, the convex function $u$
satisfies \eqref{SS}. It has been shown that up to an additive
constant $u$ satisfies the elliptic equation (see
\cite{CCY}\cite{H}\cite{HW})
\begin{equation}\aligned\label{lss}
\log \det D^2u(x)=\f12x\cdot Du(x)-u(x).
\endaligned
\end{equation}

Huang-Wang \cite{HW} and Chau-Chen-Yuan \cite{CCY} have investigated
the entire solutions to the above equation and showed that  an
entire smooth convex solution to \eqref{lss} in $\R^n$ is the
quadratic polynomial under the decay condition on Hessian of $u$.

In \cite{CCY}, Chau-Chen-Yuan introduce a natural geometric quantity
$\phi=\log\det D^2u$ which obeys a second order elliptic equation
with an ''amplifying force''.  Based on it, we consider an important
operator: the drift Laplacian operator $\mathcal{L}$, which was
introduced by Colding-Minicozzi \cite{CM1}, and we can also write
the second order equation for $\phi$ in \cite{CCY} as
$\mathcal{L}\phi=0$. This enables us to apply integral method to
prove any entire smooth proper convex solution to \eqref{lss} in
$\R^n$ is the quadratic polynomial,  Theorem \ref{bnst}, where the
case $n=1$ is simple.

It is worth to note that when $\phi$ is constant the mean curvature
of $M$ vanishes (see (8.5.7) of Chap. VIII in \cite{X}), namely, the
gradient graph of a solution $u$ to \eqref{lss} defines a space-like
minimal Lagrangian submanifold in $\ir{2n}_n$.

By thoroughly analysing the convexity of $u$, we could prove that
any solution of \eqref{lss} is proper, which is showed in Theorem
\ref{lim}. Thus, we remove the additional condition of the
corresponding results in \cite{HW} and \cite{CCY}. Precisely, we
obtain

\begin{theorem}\label{main}
Any entire smooth convex solution $u(x)$ to \eqref{lss} in $\ir{n}$
is the quadratic polynomial $u(0)+\f12\lan D^2u(0)x, x\ran$.
\end{theorem}

We also consider the corresponding problem in ambient Euclidean
space: a Lagrangian graph $\{(x,Du(x))\big|\ x\in\R^n\}$ in
$\R^{2n}$ satisfying \eqref{SS}. Now,  $u$ is an entire solution to
the following equation:
\begin{equation}\aligned\label{lls}
\arctan\la_1(x)+\cdots+\arctan\la_n(x)=\f12x\cdot Du(x)-u(x)
\endaligned
\end{equation}
where $\la_1(x),\cdots,\la_n(x)$ are the eigenvalues of the Hessian $D^2u$ of $u$ at $x\in\R^n$.
Chau-Chen-Yuan \cite{CCY} constructed a barrier function to show that the phase function
$$\Theta=\arctan\la_1(x)+\cdots+\arctan\la_n(x)$$
on this Lagrangian graph is a constant via the maximum principles. A
geometric meaning of the phase function is the summation of the all
Jordan angles of the Gauss map $\g: M\to \grs{n}{n}$ (see \cite{X}
Chap 7,  for example). They proved the following theorem.
\begin{theorem}\label{re}
If $u(x)$ is an entire smooth solution to \eqref{lls} in $\R^n$, then $u(x)$ is
the quadratic polynomial $u(0)+\f12\lan D^2u(0)x, x\ran$.
\end{theorem}

We could also derive the phase function satisfies:
$\mathcal{L}\Theta=0$. This enables us to use the integral method to
reprove the above rigidity theorem.

{\bf Acknowledgement} The authors would like to express their
sincere thanks to Jingyi Chen for his valuable comments on the first
draft of this paper.

\section{Space-like Lagrangian self-shrinkers in pseudo-Euclidean space}

Let $M=\{(x,Du(x))\big|\ x\in\R^n\}$ be a space-like submanifold satisfying \eqref{lss} in ambient space $\R^{2n}_n$ with the induced metric $g_{ij}dx_idx_j$, where $Du=(u_1,u_2,\cdots,u_n)$.
Then $g_{ij}=\p_i\p_ju=u_{ij}$, and let $(g^{ij})$ denote the inverse matrix $(g_{ij})$. We write $g=\det{g_{ij}}$ for simplicity and $\xi\cdot\e=\lan\xi,\e\ran$ for any vectors $\xi,\e\in\R^n$. By \eqref{lss}, we have
\begin{equation}\aligned\label{p1}
\p_j(\log g)=\f12u_j+\f12x_iu_{ij}-u_j=\f12x_iu_{ij}-\f12u_j,
\endaligned
\end{equation}
and
\begin{equation}\aligned\label{p2}
\p_i(\sqrt{g}g^{ij})=&\f12\sqrt{g}g^{kl}\p_ig_{kl}g^{ij}-\sqrt{g}g^{ki}\p_ig_{kl}g^{lj}\\
=&-\f12\sqrt{g}g^{kl}u_{kli}g^{ij}=-\f12\sqrt{g}g^{ij}\p_i(\log g).
\endaligned
\end{equation}
Let $\mathcal{L}$ be a differential operator defined by
$$\mathcal{L}\phi=\f1{\sqrt{g}}e^{\f14x\cdot Du}\f{\p}{\p x_i}\left(g^{ij}\sqrt{g}e^{-\f14x\cdot Du}\f{\p}{\p x_j}\phi\right),$$
for any function $\phi\in C^2(\R^n)$.
Combining \eqref{p1} and \eqref{p2}, we have
\begin{equation}\aligned\label{Lphi}
\mathcal{L}\phi=&\f1{\sqrt{g}}\p_i(g^{ij}\sqrt{g}\phi_j)+e^{\f14x\cdot Du}\p_i(e^{-\f14x\cdot Du})g^{ij}\phi_j\\
=&g^{ij}\phi_{ij}+\f1{\sqrt{g}}\p_i(g^{ij}\sqrt{g})\phi_j-\f14(u_i+x_ku_{ki})g^{ij}\phi_j\\
=&g^{ij}\phi_{ij}-\f14(x_ku_{ki}-u_i)g^{ij}\phi_j-\f14(u_i+x_ku_{ki})g^{ij}\phi_j\\
=&g^{ij}\phi_{ij}-\f12x_ku_{ki}g^{ij}\phi_j\\
=&g^{ij}\phi_{ij}-\f12x_j\phi_j.
\endaligned
\end{equation}

\begin{remark} The submanifold $M$ in $\ir{2n}_n$ is defined by $(\ir{n}, ds^2=u_{ij}dx_idx_j)$.
The operator $\mathcal{L}$  is also defined on $M$. It is precisely
the drift Laplacian $\mathcal{L}$ in the version of space-like
self-shrinkers in pseudo-Euclidean space, which  was introduced by
Colding-Minicozzi \cite{CM1} in the ambient Euclidean space.
\end{remark}

\begin{lemma}\label{Conv}
Let $\Om$ be a convex domain in $\R^n(n\ge2)$ and $u$ be a smooth proper convex function in $\Om$,
then for any $\a>0$
$$\int_{\Om}|x\cdot Du|e^{-\a u}dx<+\infty.$$
\end{lemma}
\begin{proof}
Let $\G_t=\{x\in\Om|\ u(x)=t\}$ and $\Om_t=\{x\in\Om|\ u(x)<t\}$ for
each $t\in\R$. By the convexity of $u$, we know that $\G_t\cap L$
contains two point at most, where $L$ is any line in $\R^n$. Since
$u$ is proper, then $\G_t$ is homotopic to $(n-1)$-sphere in $\R^n$,
which implies $\Om_t$ is a bounded domain enclosed by $\G_t$. Thus,
$\inf_{x\in\Om}u(x)>-\infty$ and
$\lim_{x\rightarrow\p\Om}u(x)=+\infty$. By translating $\Om$ in the
plane $\R^n$, we can assume $0\in\Om$ and $u(0)=\inf_{x\in\Om}u(x)$.
Moreover, by the convexity of $u$, there exist constants $C,\de>0$
such that for any $x\in\Om$
\begin{equation}\aligned\label{growth}
u(x)+C\ge \de|x|.
\endaligned
\end{equation}
It suffices to show
\begin{equation}\aligned\label{I1}
\int_{\Om}|Du|e^{-\be u}dx<+\infty
\endaligned
\end{equation}
holds for some $0<\be<\a$.

Set $x'=(x_1,\cdots,x_{n-1}).$ Let $$\Om'=\{x'\in\R^{n-1}|\ \exists
x_n\ s.t.\ (x',x_n)\in\Om\}.$$
For every fixed $x'\in\Om'$,
$u_{nn}(x',x_n)=\p_{x_n}\p_{x_n}u(x',x_n)$ is positive, and
$u_{n}(x',x_n)$ is monotonic increasing in $x_n$. Since
$\lim_{x\rightarrow\p\Om}u(x)=+\infty$, then there is $x_n^*$ such
that $(x',x_n^*)\in\Om$ and $u_{n}(x',x_n^*)=0$. Furthermore, we
have $x_n^1,x_n^2\in[-\infty,+\infty]$ satisfying $x_n^1<x_n^2$ and
$(x',x_n^i)\in\p\Om$ for $i=1,2$.

For each fixed $x'\in\Om'$, we have
\begin{equation}\aligned\label{I2}
\int_{(x',x_n)\in\Om}|u_n|e^{-\be u}dx_n=&-\int_{x_n^1}^{x_n^*}u_n(x',x_n)e^{-\be u(x',x_n)}dx_n+\int_{x_n^*}^{x_n^2}u_n(x',x_n)e^{-\be u(x',x_n)}dx_n\\
=&\f1{\be}\int_{x_n^1}^{x_n^*}de^{-\be u(x',x_n)}-\f1{\be}\int_{x_n^*}^{x_n^2}de^{-\be u(x',x_n)}\\
=&\f2{\be}e^{-\be u(x',x_n^*)}.
\endaligned
\end{equation}
Since $u(x',x_n^*)+C\ge\de\sqrt{|x'|^2+(x_n^*)^2}\ge\de|x'|$, then by \eqref{I2},
\begin{equation}\aligned\label{I3}
\int_{\Om}|u_n|e^{-\be u}dx=&\int_{x'\in\Om'}\int_{(x',x_n)\in\Om}|u_n|e^{-\be u}dx_ndx'=\int_{x'\in\Om'}\f2{\be}e^{-\be u(x',x_n^*)}dx'\\
\le&\int_{x'\in\Om'}\f2{\be}e^{\be C-\be\de|x'|}dx'<\infty.
\endaligned
\end{equation}
By the same way to $\{u_i\}$ for $i=1,\cdots,n-1$, we know \eqref{I1} holds. This shows the Lemma.
\end{proof}

\begin{theorem}\label{bnst}
If $\Om$ is a convex domain containing the origin in $\R^n(n\ge2)$ and $u(x)$ is a smooth proper convex solution to \eqref{lss} in $\Om$, then $\Om$ is $\R^n$ and $u(x)$ is the quadratic polynomial $u(0)+\f12\lan D^2u(0)x, x\ran$.
\end{theorem}
\begin{proof}
Let $\phi=\log g$, then by \eqref{p1}, $\phi_{ij}=\f12x_ku_{ijk}$ and
\begin{equation}\aligned\label{xphi}
g^{ij}\phi_{ij}=\f12g^{ij}x_ku_{ijk}=\f12x_k\phi_k.
\endaligned
\end{equation}
\eqref{xphi} was found by Chau-Chen-Yuan \cite{CCY}. Combining
\eqref{Lphi} and \eqref{xphi}, we have
\begin{equation}\aligned\label{Llg}
\mathcal{L}\phi=g^{ij}\phi_{ij}-\f12x_j\phi_j=0.
\endaligned
\end{equation}

Let $F$ be a positive monotonic increasing $C^1$-function on $\R$,
and $\e$ be a nonnegative Lipschitz function in $\Om$ with compact
support, both to be defined later. Using \eqref{lss} and \eqref{Llg}
and integrating by parts, we have
\begin{equation}\aligned\label{f1}
0=&-\int_{\Om}F(\phi)\e^2\mathcal{L}\phi\ e^{-\f14x\cdot Du}\sqrt{g}dx\\
=&\int_{\Om}g^{ij}\p_i\left(F(\phi)\e^2\right)\phi_je^{-\f14x\cdot Du}\sqrt{g}dx\\
=&\int_{\Om}g^{ij}\phi_i\phi_jF'\e^2e^{-\f{u}2}dx+2\int_{\Om}F(\phi)\e g^{ij}\e_i\phi_je^{-\f{u}2}dx.
\endaligned
\end{equation}
Since $u$ is proper convex, then $\lim_{x\rightarrow\p\Om}u(x)=+\infty$ and we define the set $\Om_t=\{x\in\Om|\ u(x)<t\}$ as Lemma \ref{Conv}, which is an exhaustion of the domain $\Om$. Let
\begin{eqnarray*}
   \e(x)\triangleq \left\{\begin{array}{ccc}
     1     & \quad\ \ \ {\rm{if}} \ \ \  x\in\Om_t \\ [3mm]
     t+1-u(x)    & \quad\ \ \ {\rm{if}} \ \ \  x\in\Om_{t+1}\setminus\Om_{t}\\ [3mm]
     0  & \quad\quad\  {\rm{if}} \ \ \  x\in\Om\setminus\Om_{t+1},
     \end{array}\right.
\end{eqnarray*}
and
\begin{eqnarray*}
   F(s)\triangleq \left\{\begin{array}{ccc}
     e^s     & \quad\ \ \ {\rm{if}} \ \ \  s<0 \\ [3mm]
     1    & \quad\ \ \ {\rm{if}} \ \ \  s=0\\ [3mm]
     1+\arctan s  & \quad\quad\  {\rm{if}} \ \ \  s>0.
     \end{array}\right.
\end{eqnarray*}
By \eqref{p1} and \eqref{f1}, we have
\begin{equation}\aligned\label{f2}
\int_{\Om_t}g^{ij}\phi_i\phi_jF'e^{-\f{u}2}dx\le&\int_{\Om}g^{ij}\phi_i\phi_jF'\e^2e^{-\f{u}2}dx=-2\int_{\Om}F(\phi)\e g^{ij}\e_i\phi_je^{-\f{u}2}dx\\
=&2\int_{\Om_{t+1}\setminus\Om_{t}}F(\phi)\e g^{ij}u_i(\f12x_ku_{jk}-\f12u_j)e^{-\f{u}2}dx\\
\le&\int_{\Om_{t+1}\setminus\Om_{t}}F(\phi)\e x_iu_ie^{-\f{u}2}dx\\
\le&(1+\f{\pi}2)\int_{\Om_{t+1}\setminus\Om_{t}}|x_iu_i|e^{-\f{u}2}dx.
\endaligned
\end{equation}
By Lemma \ref{Conv}, let $t$ go to infinity in \eqref{f2}, we know
$D\phi=0$ and $\phi=\log g$ is a constant in $\Om$. By the equation
\eqref{lss} and $0\in\Om$, as shown in \cite{CCY}, we know $u(x)$ is the quadratic
polynomial $u(0)+\f12\lan D^2u(0)x, x\ran$. Since
$\lim_{x\rightarrow\p\Om}u(x)=+\infty$, then $\Om=\R^n.$
\end{proof}

As for $n=1$, the equation \eqref{lss} gives the equation
$$u''=e^{\f{1}{2}xu'-u}.$$ Since $(xu'-u)'=xu''$, we have $xu'-u\ge
-u(0)$ and
$$u''(x)\ge e^{-\f{u(0)+u(x)}{2}}.$$ If
$\lim_{x\rightarrow+\infty}u(x)=C_0\in[-\infty,+\infty)$, then
$u(x)\le\max\{u(0),C_0\}$ on $[0,+\infty).$ Then $$u''(x)\ge
e^{-\f{u(0)+\max\{u(0),C_0\}}2}.$$ This means that
$$\lim_{x\rightarrow+\infty}u'(x)=+\infty,$$ which contradicts with
$\lim_{x\rightarrow+\infty}u(x)<+\infty$. A similar argument
concludes that\\ $\lim_{x\rightarrow -\infty}u(x)=+\infty$. Thus,
$$\lim_{|x|\rightarrow\infty}u(x)=+\infty.$$ Combining \eqref{growth},
we have
$$\int_\R|xu'|e^{-\f{u}2}dx<\infty.$$
Following the argument of Theorem \ref{bnst}, we could prove Theorem \ref{main} for the case $n=1$.

For proving Theorem \ref{main} completely, it suffices to remove the proper
condition  of $u(x)$ in Theorem \ref{bnst} when $\Om=\R^n$.
Now we give two lemmas on convex functions which will be used in
Theorem \ref{lim} in the case $n\ge 2$. One is an algebraic property
for the Hessian of convex functions, the other is on the size of
Lebesgue measure of a set which arises from the equation
\eqref{lss}.

\begin{lemma}\label{alg}
Let $u$ be a smooth convex function in a domain of $\R^n$. If
$\xi_1,\cdots,\xi_n$ is an arbitrary  orthonormal basis of $\R^n$,
then
$$\det D^2u\le u_{\xi_1\xi_1}u_{\xi_2\xi_2}\cdots u_{\xi_n\xi_n},$$
where $u_{\xi_i\xi_j}= \mathrm{Hessian} (u) (\xi_i,\xi_j)$ in $\ir{n}$
for $1\le i,j\le n$.
\end{lemma}
\begin{proof}
By an orthogonal transformation, we have
\begin{equation}\aligned\label{4.1}
\det D^2u=\det u_{\xi_i\xi_j}.
\endaligned
\end{equation}
Let $\a$ be a $(n-1)$-dimensional vector defined by
$(u_{\xi_1\xi_2},u_{\xi_1\xi_3},\cdots,u_{\xi_1\xi_n})$ and $A$ be a
$(n-1)\times(n-1)$ matrix $(u_{\xi_i\xi_j})_{2\le i,j\le n}$. Since
\begin{equation}\nonumber
\left(\begin{array}{cc}
1&   0\\
-\f1{u_{\xi_1\xi_1}}\a^T&  I_{n-1}  \\
\end{array}\right)\left(\begin{array}{cc}
u_{\xi_1\xi_1}&  \a \\
\a^T&  A  \\
\end{array}\right)\left(\begin{array}{cc}
1&  -\f1{u_{\xi_1\xi_1}}\a \\
0&  I_{n-1}  \\
\end{array}\right)=\left(\begin{array}{cc}
u_{\xi_1\xi_1}&  0 \\
0&  A-\f{\a^T\a}{u_{\xi_1\xi_1}} \\
\end{array}\right),
\end{equation}
then $A-\f{\a^T\a}{u_{\xi_1\xi_1}}$ is a positive definite matrix and
\begin{equation}\aligned\label{4.2}
\det D^2u=\det
u_{\xi_i\xi_j}=u_{\xi_1\xi_1}\det\left(A-\f{\a^T\a}{u_{\xi_1\xi_1}}\right)\le
u_{\xi_1\xi_1}\det(A).
\endaligned
\end{equation}
By induction,
\begin{equation}\label{4.3}
\det D^2u\le u_{\xi_1\xi_1}u_{\xi_2\xi_2}\cdots u_{\xi_n\xi_n}.
\end{equation}
\end{proof}

\begin{lemma}\label{comp}
Let $B_\de$ be an open ball with radius $\de$ and centered at the origin in $\R^m$, $v$ be a smooth convex function in $\overline{B_\de}$ with $v\big|_{\p B_\de}\le C_1$, then there is a constant $C_2>0$ depending only on $m,\de,C_1$ such that the set
$$E=\{x\in B_\de|\ e^{\f{v(x)}2}\det D^2v>C^m_2\}$$
has the measure $|E|<\f12|B_\de|$.
\end{lemma}
\begin{proof}
Suppose that the measure $|E|\ge\f12|B_\de|$ for some sufficiently
large $C_2$, and we will deduce the contradiction. Denote the open
sets
$$E_i=\{x\in B_\de\big|\ D_{ii}v(x)>C_2e^{-\f{v(x)}{2m}}\}$$ for
$i=1,\cdots,m$. By Lemma \ref{alg}, $$D_{11}vD_{22}v\cdots
D_{mm}v\ge\det D^2v,$$ then $$E\subset\bigcup_{1\le i\le m}E_i.$$ Thus,
$$\f12|B_\de|\le|E|\le|\bigcup_{1\le i\le m}E_i|\le\sum_{i=1}^m|E_i|,$$
which implies there is a $E_i$ with
$$|E_i|\ge\f1{2m} |B_\de|.$$ Without loss of generalarity
set $E_1=E_i$, then there is $$L=\{x=(x_1,\cdots,x_m)\in B_\de\big|\
x_2=y_2,\cdots,x_m=y_m\}$$ such that the measure of $L\cap E_1$ is
no less than $C_3\de$ for some constant $0<C_3<1$ depending only on
$m$.

Let $f(s)=v(s,y_2,y_3,\cdots,y_m)$, $I=\{s\in\R\big|\
(s,y_2,y_3,\cdots,y_m)\in L\}$, then $I=(-s_0,s_0)$ with
$\f{C_3\de}2\le s_0\le\de$ and $$E_1=\{|s|<s_0|\
f''(s)>C_2e^{-\f{f(s)}{2m}}\}.$$ Without loss of generality, we
select $0\le a_1<b_1<a_2<b_2<\cdots<a_N<b_N$ for some $N<\infty$
such that $L\cap
E_1\supset\bigcup_{i=1}^N(a_i,b_i)\times(y_2,\cdots,y_m)$ and
$\sum_{i=1}^N(b_i-a_i)\ge\f{C_3\de}3$.

For deducing the contradiction, we need prove $f(s_0)$ is sufficiently large as $C_2$ is sufficiently large, which violates our assumption $v\big|_{\p B_\de}\le C_1$.
Since $v$ is convex, then $\sup_{x\in \overline{B_\de}}v(x)\le C_1$ and $f''=D_{11}v>0$. By Newton-Leibnitz formula, we get
$$f(0)-C_1\le f(0)-f(-s_0)=\int_{-s_0}^0f'(s)ds\le f'(0)s_0,$$
which implies
\begin{equation}\aligned\label{f'(0)}
f'(0)\ge\f{f(0)-C_1}{s_0}.
\endaligned
\end{equation}
Let $C_4=\f1{s_0}(C_1e^{\f{C_1}{2m}}-f(0)e^{\f{f(0)}{2m}})$, which depends only on $m,\de$ and $C_1$.
Since $C_2$ is sufficiently large, then there is a $c\in(a_j,b_j)$ such that
\begin{equation}\aligned\label{cj}
\sum_{i=1}^{j-1}(b_i-a_i)+c-a_j\in\left(\f{C_4}{C_2},\f{2C_4}{C_2}\right).
\endaligned
\end{equation}
If $f'(c)<0$, then $f(0)\ge f(s)$ for $s\in[0,c]$. Combining \eqref{f'(0)}, \eqref{cj} and the definition of $E_1$ and $C_4$, we have
\begin{equation}\aligned
f'(c)=&f'(0)+\int_0^cf''(s)ds\ge f'(0)+\sum_{i=1}^{j-1}\int^{b_i}_{a_i}f''(s)ds+\int^{c}_{a_j}f''(s)ds\\
\ge&f'(0)+\sum_{i=1}^{j-1}\int^{b_i}_{a_i}C_2e^{-\f{f(s)}{2m}}ds+\int^{c}_{a_j}C_2e^{-\f{f(s)}{2m}}ds\\
\ge&f'(0)+\sum_{i=1}^{j-1}\int^{b_i}_{a_i}C_2e^{-\f{f(0)}{2m}}ds+\int^{c}_{a_j}C_2e^{-\f{f(0)}{2m}}ds\\
\ge&\f{f(0)-C_1}{s_0}+C_4e^{-\f{f(0)}{2m}}=e^{-\f{f(0)}{2m}}\left(C_4+\f1{s_0}\big(f(0)e^{\f{f(0)}{2m}}-C_1e^{\f{f(0)}{2m}}\big)\right)\\
\ge&e^{-\f{f(0)}{2m}}\left(C_4+\f1{s_0}\big(f(0)e^{\f{f(0)}{2m}}-C_1e^{\f{C_1}{2m}}\big)\right)=0.
\endaligned
\end{equation}
Thus, $f'(c)\ge0$. Together with $f''>0$, we have
\begin{equation}\aligned\label{f'}
0\le f'(s_1)\le f'(s_2)\ \mathrm{and}\ f(s_1)\le f(s_2)\quad \mathrm{for}\ c\le s_1\le s_2\le s_0.
\endaligned
\end{equation}
Denote $\de_j=b_j-c$ and $\de_k=b_k-a_k$ for $k=j+1,\cdots,N$. By \eqref{f'} and the definition of $E_1$, for $t\in(c,b_j)$ we obtain
$$f'(t)=f'(c)+\int_c^tf''(s)ds\ge C_2\int_c^te^{-\f{f(s)}{2m}}ds\ge C_2(t-c)e^{-\f{f(t)}{2m}},$$
then
\begin{equation}\aligned\nonumber
e^{\f{f(t)}{2m}}=e^{\f{f(c)}{2m}}+\int_c^t\f{f'(s)}{2m}e^{\f{f(s)}{2m}}ds\ge\int_c^t\f{C_2}{2m}(s-c)ds=\f{C_2}{4m}(t-c)^2.
\endaligned
\end{equation}
So we claim
\begin{equation}\aligned\label{claim}
f'(b_k)\ge C_2\sum_{i=j}^k\de_ie^{-\f{f(b_k)}{2m}},\ \ \ \mathrm{and} \ \ \ e^{\f{f(b_k)}{2m}}\ge \f{C_2}{4m}\bigg(\sum_{i=j}^k\de_i\bigg)^2\quad \mathrm{for}\ k=j,\cdots,N.
\endaligned
\end{equation}
If \eqref{claim} holds for some $k<N$, then $f'(a_{k+1})\ge f'(b_k)$ and $f(a_{k+1})\ge f(b_k)$ by \eqref{f'}. For any $t\in(a_{k+1},b_{k+1})$, we get
\begin{equation}\aligned
f'(t)=&f'(a_{k+1})+\int_{a_{k+1}}^tf''(s)ds\ge C_2\sum_{i=j}^k\de_ie^{-\f{f(b_k)}{2m}}+C_2\int_{a_{k+1}}^te^{-\f{f(s)}{2m}}ds\\
\ge& C_2\bigg(t-a_{k+1}+\sum_{i=j}^k\de_i\bigg)e^{-\f{f(t)}{2m}},
\endaligned
\end{equation}
and
\begin{equation}\aligned
e^{\f{f(t)}{2m}}=&e^{\f{f(a_{k+1})}{2m}}+\int_{a_{k+1}}^t\f{f'(s)}{2m}e^{\f{f(s)}{2m}}ds\\
\ge& \f{C_2}{4m}\bigg(\sum_{i=j}^k\de_i\bigg)^2+\int_{a_{k+1}}^t\f{C_2}{2m}\bigg(s-a_{k+1}+\sum_{i=j}^k\de_i\bigg)ds\\
=&\f{C_2}{4m}\bigg(t-a_{k+1}+\sum_{i=j}^k\de_i\bigg)^2.
\endaligned
\end{equation}
By induction, we complete this claim. Combining the selection of $a_i,b_i$ and \eqref{cj}\eqref{f'}\eqref{claim}, we conclude
\begin{equation}\aligned\nonumber
C_1\ge f(s_0)\ge f(b_N)\ge2m\log\f{C_2}{4m}+4m\log\sum_{i=j}^N\de_i\ge2m\log\f{C_2}{4m}+4m\log(\f{C_3\de}3-\f{2C_4}{C_2}),
\endaligned
\end{equation}
which is impossible for sufficiently large $C_2$.
\end{proof}


\begin{theorem}\label{lim}
Any entire smooth convex solution $u$ to \eqref{lss} in $\ir{n}$ is
proper.
\end{theorem}
\begin{proof}
To prove the result, it suffices to show $\lim_{|x|\rightarrow\infty}u(x)=+\infty$ for $n\ge2$. 
Let $B_r^n$ be an open ball in $\R^n$ with radius $r$
and centered at the origin. Suppose that
\begin{equation}\aligned\label{bdu}
\liminf_{|x|\rightarrow\infty}u(x)<+\infty.
\endaligned
\end{equation}
Since $\f{\p}{\p r}\bigg(r\lan\be,Du(r\be)\ran-u(r\be)\bigg)=ru_{ij}\be_i\be_j>0$ for every $\be=(\be_1,\cdots,\be_n)\in\S^{n-1}(1)$, then
$$r\p_ru(r\be)-u(r\be)\ge-u(0),$$
and
$$\left(\f{u(r\be)-u(0)}r\right)'=\f{r\p_ru(r\be)-u(r\be)+u(0)}{r^2}\ge0.$$
So $\lim_{r\rightarrow\infty}\f{u(r\be)}r$ always exists (may be infinity)
and is denoted by $\k_\be$. Let $\La=\{\be\in\S^{n-1}(1)|\ \k_\be\le0\}$.
If $\La=\emptyset$, then for any $\be\in\S^{n-1}(1)$, there is a $r_\be>0$
such that $u(r_\be \be)-u(0)\ge\f12\tilde{\k}_\be r_\be$, where $\tilde{\k}_\be=\min\{\k_\be,1\}>0$.
By the continuity of $u$, there is an open domain $S_\be\subset\S^{n-1}(1)$ containing $\be$ such that
$u(r_\be \g)-u(0)\ge\f14\tilde{\k}_\be r_\be$ for each $\g\in S_\be$. Since $u$ is convex,
then $\p_ru(r\g)\ge\f14\tilde{\k}_\be$ for $r\ge r_\be$, which implies
$$u(r\g)-u(0)=\int_{r_\be}^r\p_su(s\g)ds+u(r_\be \g)-u(0)
\ge\f14\tilde{\k}_\be (r-r_\be)+\f14\tilde{\k}_\be
r_\be=\f14\tilde{\k}_\be r$$ for each $\g\in S_\be$ and $r\ge
r_\be$. By the finite cover property, there is a sequence
$\{\tilde{\be}_i\}_{i=1}^N$ such that
$\S^{n-1}(1)\subset\bigcup_{1\le i\le N} S_{\tilde{\be}_i}.$ Let
$r^*=\max_{1\le i\le N}r_{\tilde{\be}_i}$ and $\k^*=\min_{1\le i\le
N}\tilde{\k}_{\tilde{\be}_i}>0$, then for any $\be\in\S^{n-1}(1)$
and $r\ge r^*$, we have $u(r\be)-u(0)\ge\f14\k^*r.$ This contradicts
with \eqref{bdu}. Therefore, $\La$ is nonempty.

There is a sequence $\{\bar{\be}_i\}\subset\La$ such that
$$\lim_{i\rightarrow\infty}\k_{\bar{\be}_i}=\inf_{\be\in\La}\k_{\be}.$$
And we can assume $\lim_{i\rightarrow\infty}\bar{\be}_i=\th$ for some $\th=(\th_1,\cdots,\th_n)\in\S^{n-1}(1)$. For every fixed $r>0$, there is a $i_0>0$ such that for all $i\ge i_0$, $u(r\bar{\be}_i)\ge u(r\th)-1$. Then
$$0\ge\k_{\bar{\be}_i}\ge\f{u(r\bar{\be}_i)-u(0)}r\ge\f{u(r\th)-u(0)-1}r.$$
Hence $u(r\th)\le u(0)+1$, and
$$\lim_{i\rightarrow\infty}\k_{\bar{\be}_i}\ge\lim_{r\rightarrow\infty}\f{u(r\th)-u(0)-1}r=\k_\th.$$
Therefore $\k_\th=\inf_{\be\in\La}\k_{\be}\le0$. Let $\k=\k_\th$ for simplicity. For each $\be\in\S^{n-1}(1)$, we obtain
\begin{equation}\aligned\label{kappa}
\k=\lim_{r\rightarrow\infty}\f{u(r\th)}r=\lim_{r\rightarrow\infty}\lan\th,Du(r\th)\ran
\le\lim_{r\rightarrow\infty}\f{u(r\be)}r.
\endaligned
\end{equation}
Let
$$U=\{x\in\R^n\big|\ u(x)<\k\lan\th,x\ran+u(0)\}.$$
Since $u$ is an entire convex function in $\R^n$, then $U$ is a
convex domain in $\R^n$. The definition of $\k$ implies $r\th\in U$
for any $r>0$. We then can find a slim column region around the ray
$r\th$ inside the convex domain $U$. Precisely,  there exist
$r_0,\de>0$ such that
$$\sr{C}_\th\triangleq\{r\th+\a\in\R^n\big|\ r\ge r_0,\ \a\bot\th\ \mathrm{and}\ |\a|<\de\}\subset U.$$
Let $$\sr{S}_r=\{r\th+\a\big|\ \a\bot\th\ \mathrm{and}\ |\a|<\de\}$$
be a slice of $\sr{C}_\th.$ Let $u_{\th}(r\th+\a)=\f{\p}{\p
r}u(r\th+\a)=\lan\th, Du(r\th+\a)\ran$ denote the $\th-$directional
derivative of $u$ and $$u_{\th\th}(r\th+\a)=\f{\p^2}{\p
r^2}u(r\th+\a)=\sum_{i,j}u_{ij}(r\th+\a)\th_i\th_j.$$ By
$\sr{C}_\th\subset U$ and \eqref{kappa}, we conclude
$\lim_{r\rightarrow\infty}u_\th(r\th+\a)=\k$ for any
$\a\bot\th,|\a|<\de$. We don't have the pointwise estimate for
$u_{\th\th}$ in $\sr{C}_\th$, but have the following integral
estimate
\begin{equation}\aligned\label{dense}
\int_r^\infty\int_{\sr{S}_s}u_{\th\th}dV_{\sr{S}_s}ds=&\int_r^\infty\int_{\a\bot\th,|\a|<\de} u_{\th\th}(s\th+\a)dV_{\a}ds\\
=&\int_{\a\bot\th,|\a|<\de}\int_r^\infty u_{\th\th}(s\th+\a)dsdV_{\a}\\
=&\int_{\a\bot\th,|\a|<\de}\big(\k-u_\th(r\th+\a)\big)dV_{\a}<\infty.\\
\endaligned
\end{equation}
Let $\omega_{n-1}$ be the standard volume of $(n-1)$-dimensional
unit balls. From \eqref{dense}, we can find a sequence
$\{r_i\}_{i=1}^\infty\subset\R$ with
$\lim_{i\rightarrow\infty}r_i=+\infty$ such that the open set
$$\widetilde{\sr{S}}_{r_i}\triangleq\{x\in \sr{S}_{r_i}\big|\ u_{\th\th}(x)<\f1{r_i}\}$$
has measure
$$|\widetilde{\sr{S}}_{r_i}|\ge\f12|\sr{S}_{r_i}|=\f12\omega_{n-1}\de^{n-1}.$$
Here, the factor $\f12$ is not essential and could be replaced by
any positive constant which is less than 1.

Since $\f{\p}{\p r}\bigg(r\lan\be,Du(r\be)\ran-u(r\be)\bigg)=ru_{ij}\be_i\be_j>0$ for every $\be\in\S^{n-1}(1)$, then
$$x\cdot Du(x)-u(x)\ge-u(0),$$
and
\begin{equation}\aligned\label{est}
\det D^2u(x)=e^{\f12x\cdot Du(x)-u(x)}\ge e^{-\f{u(0)+u(x)}2}.
\endaligned
\end{equation}

Let $D^2_{\sr{S}}u$ be the $(n-1)$-Hessian matrix of $u$ in
$\sr{S}_r$ for each $r\ge r_0$, then $D^2_{\sr{S}}u>0$. By
\eqref{4.2} and \eqref{est}, we get
$$\det D^2_\sr{S}u\ge u_{\th\th}^{-1}\det D^2u\ge u_{\th\th}^{-1}e^{-\f{u(0)+u(x)}2}.$$
The definition of $U$ implies that $u(x)\le u(0)$ for any $x\in\sr{S}_{r_i}\subset U$.
Combining the measure $|\widetilde{\sr{S}}_{r_i}|=|\{x\in
\sr{S}_{r_i}\big|\ u_{\th\th}^{-1}(x)>r_i\}|\ge\f12|\sr{S}_{r_i}|$
and Lemma \ref{comp}, we arrive at a contradiction if $i$ goes to
infinity. Therefore, $\lim_{|x|\rightarrow\infty}u(x)=+\infty$ when
$n\ge2$. We complete the proof.
\end{proof}

\noindent {\it Proof of Theorem \ref{main}.} Noting the case $n=1$ and Combining Theorem \ref{bnst} and Theorem \ref{lim},
we finish the proof. \qed

Let $a>0$, $c$ be constant numbers and $b\in\R^n$ be a constant vector. The entire solution to the following general type equation
$$\log\det D^2u(x)=a\big(\f12x\cdot Du(x)-u(x)\big)+b\cdot x+c$$ is a quadratic polynomial.
In fact, let $w(x)=au(x)-2b\cdot x-c-n\log a$, then $w$ satisfies the equation \eqref{lss}.

\section{Application to other equations}

Let's prove Theorem \ref{re} by the integral method, which is
similar to the previous section.

\noindent {\it Proof of Theorem \ref{re}.} Let $M$ be a Lagrangian submanifold satisfying \eqref{lls} in $\R^{2n}$ with induced metric $g_{ij}dx_idx_j$.
Then $g_{ij}=\de_{ij}+\sum_ku_{ik}u_{jk}$, and denote $g=\det{g_{ij}}$ for short. Then
\begin{equation}\aligned\label{e1}
\p_i(g^{ij}\sqrt{g})=&\f12\sqrt{g}g^{kl}\p_ig_{kl}g^{ij}-\sqrt{g}g^{ki}\p_ig_{kl}g^{lj}\\
=&\f12\sqrt{g}g^{kl}g^{ij}(u_{ks}u_{lsi}+u_{ksi}u_{ls})-\sqrt{g}g^{ki}g^{lj}(u_{ks}u_{lsi}+u_{ksi}u_{ls})\\
=&-\sqrt{g}g^{kl}g^{ij}u_{kls}u_{is}.
\endaligned
\end{equation}
Define the differential operator $\mathcal{L}$ on $C^2(\R^n)$ by
$$\mathcal{L}\phi=\f1{\sqrt{g}}e^{\f{|x|^2+|Du|^2}4}\f{\p}{\p x_i}\left(g^{ij}\sqrt{g}e^{-\f{|x|^2+|Du|^2}4}\f{\p}{\p x_j}\phi\right),$$
which is the same as the drift Laplacian in \cite{CM1}.

Let $\Theta=\arctan\la_1(x)+\cdots+\arctan\la_n(x)$, which is the phase function on Lagrangian submanifold $M\in\R^{2n}$.
By \cite{CCY},
\begin{equation}\aligned\label{e2}
\Theta_k=g^{ij}u_{ijk}=-\f12u_k+\f12x\cdot Du_k
\endaligned
\end{equation}
and we have
\begin{equation}\aligned\label{e3}
\mathcal{L}\Theta=&g^{ij}\Theta_{ij}+\f1{\sqrt{g}}\p_i(g^{ij}\sqrt{g})\Theta_j-\f12g^{ij}(x_i+u_ku_{ki})\Theta_j\\
=&g^{ij}\Theta_{ij}-g^{kl}g^{ij}u_{kls}u_{is}\Theta_j-\f12g^{ij}(x_k\de_{ki}+u_ku_{ki})\Theta_j\\
=&g^{ij}\Theta_{ij}+g^{ij}\big(\f12u_s-\f12x_ku_{ks}\big)u_{is}\Theta_j-\f12g^{ij}\big(x_k(g_{ki}-u_{ks}u_{is})+u_ku_{ki}\big)\Theta_j\\
=&g^{ij}\Theta_{ij}-\f12g^{ij}x_kg_{ki}\Theta_j=g^{ij}\Theta_{ij}-\f12x_j\Theta_j.\\
\endaligned
\end{equation}
By \eqref{e2}, $\Theta_{kl}=\f12x_su_{skl}$. Then $g^{kl}\Theta_{kl}=g^{kl}\f12x_su_{skl}=\f12x_j\Theta_j$ (see also \cite{CCY}), which implies
\begin{equation}\aligned\label{e4}
\mathcal{L}\Theta=0.
\endaligned
\end{equation}
Let $\n$ and $d\mu$ be Levi-Civita connection and volume element of $M$ with respect to the metric
$g_{ij}dx_idx_j$, and $\r=e^{-\f{|x|^2+|Du|^2}4}$. If $\e$ is a smooth function in $M$ with compact support,
then by integral by parts we have
\begin{equation}\aligned\label{e5}
0=&-\int_M\e^2\Theta\mathcal{L}\Theta \r d\mu=2\int_M\e\Theta\n\e\cdot\n\Theta \r d\mu+\int_M|\n\Theta|^2\e^2\r d\mu\\
\ge&-2\int_M|\n\e|^2\Theta^2\r d\mu-\f12\int_M|\n\Theta|^2\e^2\r d\mu+\int_M|\n\Theta|^2\e^2\r d\mu,
\endaligned
\end{equation}
which implies
\begin{equation}\aligned\label{e6}
\int_M|\n\Theta|^2\e^2\r d\mu\le4\int_M|\n\e|^2\Theta^2\r d\mu.
\endaligned
\end{equation}
Since $\Theta$ is a bounded function and $M$ has Euclidean volume
growth \cite{DX}, then we obtain $\Theta$ is a constant. Then, as
shown in \cite{CCY}, we obtain Theorem \ref{re}. \qed

Now, let's consider another equation.
If $v$ is a smooth subharmonic function on $\R^n$ satisfying
\begin{equation}\aligned\label{2.1}
\log\De v=\f12x\cdot Dv-v.
\endaligned
\end{equation}
Let $\phi=\log\De v$, then $\phi_i=-\f12v_i+\f12x_jv_{ij}$ and $\phi_{ii}=\f12x_jv_{iij}$. We have
\begin{equation}\aligned\label{2.2}
\De\phi=\f12x_j\p_j(\De v)=\f12e^\phi x\cdot D\phi.
\endaligned
\end{equation}
\begin{theorem}
Let $\phi(x)$ be an entire smooth solution to \eqref{2.2} in $\R^n$ and $\e$ be a Lipschitz function in $\R^n$ with compact support
and $\e\big|_{B_r}\equiv1$. If
$$\lim_{r\rightarrow+\infty}\int_{\R^n\setminus B_r}\f{|D\e|^2}{|x|^2}e^{-\phi}e^{-\f{|x|^2}4e^\phi}=0,$$
then $\phi$ is a constant.
\end{theorem}
\begin{proof}
Let $\e$ be a Lipschitz function on $\R^n$ with compact support and $\e\mid_{B_r}\equiv1$, then we multiply $\e^2e^{-\f{|x|^2}4e^\phi}$ on both sides of \eqref{2.2} and integral by parts,
\begin{equation}\aligned\label{2.3}
\int_{\R^n}\f12x\cdot D\phi &e^\phi\e^2e^{-\f{|x|^2}4e^\phi}=-\int_{\R^n}D\phi\cdot D(\e^2e^{-\f{|x|^2}4e^\phi})\\
=&\int_{\R^n}D\phi\cdot(\f12x+\f{|x|^2}4D\phi)e^\phi\e^2e^{-\f{|x|^2}4e^\phi}-2\int_{\R^n}\e D\e\cdot D\phi e^{-\f{|x|^2}4e^\phi}.
\endaligned
\end{equation}
Hence we have
\begin{equation}\aligned\label{2.4}
\f14\int_{\R^n}|x|^2|D\phi|^2&e^\phi\e^2e^{-\f{|x|^2}4e^\phi}=2\int_{\R^n}\e D\e\cdot D\phi e^{-\f{|x|^2}4e^\phi}\\
\le&\f14\int_{\R^n\setminus B_r}|x|^2|D\phi|^2e^\phi\e^2e^{-\f{|x|^2}4e^\phi}+4\int_{\R^n\setminus B_r}\f{|D\e|^2}{|x|^2}e^{-\phi}e^{-\f{|x|^2}4e^\phi},
\endaligned
\end{equation}
then
\begin{equation}\aligned\label{2.5}
\int_{B_r}|x|^2|D\phi|^2e^\phi e^{-\f{|x|^2}4e^\phi}\le16\int_{\R^n\setminus B_r}\f{|D\e|^2}{|x|^2}e^{-\phi}e^{-\f{|x|^2}4e^\phi}.
\endaligned
\end{equation}
Let $r\rightarrow\infty$, then \eqref{2.5} implies $\phi$ is a constant.
\end{proof}

For the case $n\ge3$, let
\begin{eqnarray*}
   \e(x)\triangleq \left\{\begin{array}{ccc}
     1     & \quad\ \ \ {\rm{if}} \ \ \  x\in B_r \\ [3mm]
     2-\f{|x|}{r}    & \quad\ \ \ {\rm{if}} \ \ \  x\in B_{2r}\setminus B_r\\ [3mm]
     0  & \quad\quad\  {\rm{if}} \ \ \  x\in\R^n\setminus B_{2r}.
     \end{array}\right.
\end{eqnarray*}
If $e^{\phi(x)}\ge4(n-2)\f{\log|x|}{|x|^2}$ for $|x|\ge r$, then
\begin{equation}\aligned\label{2.7}
\int_{\R^n\setminus B_r}\f{|D\e|^2}{|x|^2}e^{-\phi}e^{-\f{|x|^2}4e^\phi}\le\f1{4(n-2)}\int_{B_{2r}\setminus B_r}\f1{r^2|x|^{n-2}\log|x|}\le\f{C_n}{\log r}.
\endaligned
\end{equation}
Here, $C_n$ is a positive constant depending only on $n$.

For the case $n=2$, let
\begin{eqnarray*}
   \e(x)\triangleq \left\{\begin{array}{ccc}
     1     & \quad\ \ \ {\rm{if}} \ \ \  x\in B_r \\ [3mm]
     2-\f{\log\log|x|}{\log\log r}    & \quad\ \ \ {\rm{if}} \ \ \  x\in B_{r^{\log r}}\setminus B_r\\ [3mm]
     0  & \quad\quad\  {\rm{if}} \ \ \  x\in\R^n\setminus B_{r^{\log r}}.
     \end{array}\right.
\end{eqnarray*}
If $|x|^2\log|x|e^\phi\ge C>0$ for $|x|\ge r\ge e$, then
\begin{equation}\aligned\label{2.6}
\int_{\R^2\setminus B_r}\f{|D\e|^2}{|x|^2}e^{-\phi}e^{-\f{|x|^2}4e^\phi}\le\f1{C}\int_{B_{r^{\log r}}\setminus B_r}\f{\log|x|}{|x|^2(\log|x|)^2(\log\log r)^2}=\f{2\pi}{C\log\log r}.
\endaligned
\end{equation}
Hence, $\phi$ is a constant. By \eqref{2.1}, as
shown in \cite{CCY}, $v(x)$ is
the quadratic polynomial $v(0)+\f12\lan D^2v(0)x, x\ran$. For $n=2$, up to an additive constant \eqref{2.1} is equivalent to
$$\log\det\partial\bar{\partial}v(x)=\f12x\cdot Dv(x)-v(x).$$
Thus, our condition $\Delta v\ge \f{C}{|x|^2\log|x|}$ for any $C>0$ as $|x|\rightarrow\infty$ is weaker than $\partial\bar{\partial}v(x)\ge\f{1+\de}{2|x|^2}I$ for any $\de>0$ as $|x|\rightarrow\infty$ in \cite{CCY}.

\begin{remark}
We don't know whether every entire smooth subharmonic solution to \eqref{2.1} is the quadratic polynomial
$v(0)+\f12\lan D^2v(0)x, x\ran$.
Here, we provide a function $\phi(x)=\log(2n-4)-2\log|x|$ for $n\ge3$ and
$x\in\R^n\setminus\{0\}$, which satisfies \eqref{2.2} in $\R^n\setminus\{0\}$.
\end{remark}

\bibliographystyle{amsplain}

\end{document}